\documentclass[10pt]{amsart}
\usepackage[utf8]{inputenc}

\usepackage{amsfonts}
\usepackage{amssymb}
\usepackage{accents}
\usepackage{tikz-cd}
\usepackage{amsthm}
\usepackage{comment}
\usepackage{hyperref}
\usepackage{graphicx}
\usepackage[backend=biber, 
style=alphabetic,
sorting=nyt,
maxnames=99,
maxalphanames=5
]{biblatex}
\newtheorem{lemma}{Lemma}[section]

\newtheorem{theorem}{Theorem}[section]
\newtheorem{corollary}{Corollary}[section]

\theoremstyle{definition}

\newtheorem{definition}{Definition}[section]
\theoremstyle{remark}

\newtheorem{remark}{Remark}[section]
\theoremstyle{remark}
\newtheorem{example}{Example}[section]
\numberwithin{equation}{section}

\newcommand{\D}{\mathbb{D}}

\newcommand{\mc}[1]{\mathcal{#1}}
\newcommand{\bb}[1]{\mathbb{#1}}
\newcommand{\mf}[1]{\mathfrak{#1}}
\newcommand{\g}{\Gamma}
\newcommand{\hol}{\text{Hol}}
\renewcommand{\hom}{\text{Hom}}
\newcommand{\crit}[1]{\text{Crit}{#1}}

\newcommand{\cl}[1]{\text{Cliff}(\mathbb{P}^{#1})}
\newcommand{\T}{\mathbb{T}}
\newcommand{\p}{\mathbb{P}}

\newcommand{\0}{\mathcal{O}}

\newcommand{\I}{\text{Ind}}
\newcommand{\e}{\text{Edge}}
\renewcommand{\v}{\text{Vert}}
\renewcommand{\t}[1]{\text{#1}}

\newcommand{\lag}{\text{Lag}}

\newcommand{\ham}{\text{Ham}}

\author{Douglas Schultz}

\title{Transversality via gauge transformations for pseudo holomorphic disks in projectivized vector bundles}
\begin{document}

\begin{abstract}
We show that one can achieve transversality for lifts of holomorphic disks to a projectivized vector bundle by locally enlarging the structure group and considering the action of gauge transformations on the almost complex structure, which eases computation considerably. As an application we prove a special case of the Arnold-Givental conjecture.
\end{abstract}
\maketitle
\setcounter{tocdepth}{2}
\tableofcontents
\section{Introduction}
This paper is part of a program of computing Lagrangian Floer theory in symplectic fibrations; namely, one has a compact fiber bundle of symplectic manifolds
$$F\rightarrow E\rightarrow B$$
for which the the symplectic form on the total space restricts to the fiber form and is augmented by a pullback of the base form in the horizontal direction. One can imagine having a fibration of Lagrangians
$$L_F\rightarrow L\rightarrow L_B$$
and wanting to say something about that Lagrangian Floer theory of $L$ given that of $L_F$ and $L_B$. The author's previous work in \cite{fiberedpotential} computed the number of low energy Maslov two disks through a generic point in $L$ when the symplectic fibration was trivial above $L_B$ and $L\cong L_B\times L_F$ (and under some additional monotonicity resp. rationality assumptions on $(F,L_F)$ resp. $(B,L_B)$). The goal of this paper is to extend this result to certain classes of fibrations in two ways: The first of which is by including Lagrangians that are not products in the given symplectic fibration, and the second is by including the high energy disks in our count. We give a basic example:
\begin{example}[Example \ref{realprojectivebundles}]

On the fiber bundle $\mc{O}\oplus \mc{O}_k\xrightarrow{\pi} \bb{CP}^n$ we choose a Hermitian metric and a connection whose parallel transport has values in $U(2)$ of the fibers. We get an induced connection and fiberwise symplectic form on $\bb{P}(\mc{O}\oplus \mc{O}_k)$ to which we can apply Guillemin-Lehrman-Sternberg's \cite{guillemin} construction to obtain a closed two form $a$ that restricts to the fiber form. The symplectic structure on $\bb{P}(\mc{O}\oplus \mc{O}_k)$ is then given by $\epsilon a +\pi^*\omega_{\bb{CP}^n}$ where the $\epsilon$ prevents the symplectic form from degenerating in the horizontal direction (a \emph{weak coupling form}).

In general, there are families of almost complex structures on the total space that fiberwise agree with a chosen one and make the projection holomorphic. From the definition of the line bundles considered, there is an anti-holomorphic involution $\tau$ on $\bb{P}(\mc{O}\oplus \mc{O}_k)$ that lifts the the involution on $\bb{CP}^n$ (see the full example \ref{realprojectivebundles}), from which we get a fibered Lagrangian
$$\bb{RP}^1\rightarrow L\rightarrow \bb{RP}^n$$
Let $u:(D,\partial D)\rightarrow (\bb{CP}^n,\bb{RP}^n)$ be a holomorphic disk for the integrable toric $J_{\bb{CP}^n}$. The pullback bundle is symplectomorphic to a pullback of the Hirzebruch surface under the hemisphere map. $L$ must be invariant under parallel transport [Lemma \ref{fiberedlagrangianlemma}], so $u\vert_{\partial D}^*L$ is of the form
\begin{equation}
L:=\left\lbrace \Big( [x_0,e^{ki\theta/2}x_1],\theta \Big): x_i\in \bb{R}\right\rbrace \subset \bb{CP}^1\times S^1
\end{equation}
If we choose the $J$ on the total space to agree with the standard complex structure on the fibers and to preserve the Hirzebruch connection (see Lemma \ref{uniqueacs}), some obvious holomorphic lifts of $u$ look like
\begin{gather}
\hat{v}_{[0,1]}(z)=[0,1]\label{constantsectiontopintro}\\
\hat{v}_{[1,0]}(z)=[1,0]\label{constantsectionbotintro}\\
\hat{v}_{[x_0,x_1]}(z)=[x_0,z^{k/2}\cdot x_1] \qquad x_0\neq 0,\, k\text{ even} \label{non-constsectionintro}
\end{gather}
We would like to include some (perturbation) of these disks in the count of Maslov index 2 configurations evaluating to a generic point.

\end{example}

In order to say anything about the count of Maslov 2 configurations given the above information, one needs some sort of transversality. The main result in this paper is that for pullback bundles we can achieve transversality by enlarging the structure group. Namely, the projectivization of a vector bundle with Hermitian connection has a structure group which includes into $U(n)$ and descends to an inclusion into $PU(n)$ on the fiberwise projectivization. Take an almost complex structure $J$ on the total space to agree with fiberwise integrable $PU(n)$-invariant structure. Consider the action of the $PU(n)$-valued gauge transformations on the symplectic connection/almost complex structure on the pullback bundle:

\begin{theorem}[Theorem \ref{liftthm}]
Let $A\in H_2( u^*E,u^*L, \bb{Z})$ be a relative section class and $p\in L$ a generic point. There is a comeager set of gauge transformations $\mf{G}^{reg}\subset \mf{G}$ such that the moduli space
$$\mc{M}_{j,\mc{G}^*J}(u^*E,u^*L,A,p)$$
of $\mc{G}^*J$-holomorphic sections for $\mc{G}\in \mf{G}^{reg}$ in the class $A$ is a smooth manifold of expected dimension
$$\dim \mc{M}_{j,\mc{G}^*J}(u^*E,u^*L,A,p)=\I (A)$$
\end{theorem}

A basic computation [Lemma \ref{hologaugelemma}] shows that such gauge transformation preserve holomorphic sections. Thus, one can use the sections \ref{constantsectiontopintro}\ref{constantsectionbotintro}\ref{non-constsectionintro} to count the Maslov index 2 configurations through a generic point.

As an application, we show a special case of the Arnold-Givental conjecture.
\begin{theorem}
Let $\bb{P}(\mc{V})\rightarrow \bb{CP}^n$ be the projectivization of a vector bundle with a lift $\tau$ of the anti-symplectic involution on $\bb{CP}^n$. Then for $L=\text{Fix}(\tau)$ and any Hamiltonian isotopy $\phi^t:E\times \bb{R}\rightarrow E$, we have
$$\#\phi^1(L)\cap L\geq \sum_i \dim H^i(L,\bb{Z}_2)$$
\end{theorem}
The open Gromov-Witten invariants of the fixed point set of an anti-symplectic involution in a rational symplectic manifold are discussed in Section 5.4 \cite{CW1}. However, we give a short proof of this fact to demonstrate the ease of computation due to Theorem \ref{liftthm}. We provide an explicit computation using the aforementioned special almost complex structures, and relate the Floer cohomology groups to those computed with a general divisoral perturbation system as in \cite{CW1,CW2}.
\subsection{Outline}
Section \ref{symplecticbackgroundsection} is devoted to the construction of a symplectic manifold from a vector bundle.
Section \ref{holomorphicsectionssection} contains the language and proof of Theorem \ref{liftthm}.
Section \ref{transection} provides the modifications to previous results of the author's for transversality and compactness of general configurations in fiber bundles.
Section \ref{fukayaalgebrasection} defines the Fukaya algebra of a Lagrangian and the obstruction to Floer cohomology
Section \ref{examplesection} Contains some examples and the proof of Theorem \ref{agthm}.

\section{Symplectic manifolds from projectivized vector bundles}\label{symplecticbackgroundsection}
Let $B$ be a rational symplectic manifold, and let
$$V\rightarrow \mc{V}\rightarrow (B,\omega_B)$$
be a rank $m+1$ complex vector bundle equipped with a Hermitian metric $h$ and a compatible fiberwise integrable complex structure. We assume that there is a connection on $V$ whose holonomy lies in $U(m+1)$ of each fiber. The imaginary part of $h$ defines a global alternating two-form which restricts to a symplectic form on each fiber. With respect to an $S^1$ action we can form the fiberwise reduction, denoted
$$\bb{CP}^n\rightarrow \bb{P}(\mc{V})\rightarrow B$$
with a fiberwise symplectic two-form $a_h$ induced from the reduction of $\mc{I}m (h)$. The connection induces a splitting of the tangent space 
$$T\bb{P}^m\oplus T\bb{P}^{m,\perp a_h}=:TF\oplus H_h$$
which prescribes $PU_{m+1}$-valued parallel transport. We will often denote the total space by $E$ and the fiber by $F$

To this connection we can associate a closed connnection form via the Guilleman-Lerman-Sternberg \cite{guillemin} minimal coupling construction: Necessarily we have the \emph{curvature identity} for a closed connection form
\begin{equation}\label{curvatureidentity}
d a(v^\sharp,w^\sharp)=\iota_{[v^\sharp,w^\sharp]^{vert}}a \qquad mod\, B
\end{equation}
[1.3 \cite{guillemin}] where $v^\sharp$ is the horizontal lift of a tangent vector from $B$ and $[\bullet^\sharp,\bullet^\sharp]^{vert}$ is the projection of the Lie bracket onto $TF$. We take the value of the \emph{minimal coupling form} on $v,w\in \Lambda^2 H$ to be the unique zero-average Hamiltonian associated to \ref{curvatureidentity}, and to otherwise agree with $\mc{I}m h$. Such a form is closed [Theorem 1.4.1 \cite{guillemin}] The total space of the symplectic fibration becomes a symplectic manifold if we take
$$\omega_{\epsilon,a}:=\epsilon a +\pi^*\omega_B$$
for $\epsilon <<1$. Usually one can start with a symplectic $E$ and make it into a symplectic fibration without changing the symplectic form substantially.

A Lagrangian in this context is the following:
\begin{definition}
A \emph{fibered Lagrangian} $L$ is a fiber sub-bundle of Lagrangians
$$L_F\rightarrow L\xrightarrow{\pi} L_B$$
with $L_F$ monotone and $L_B$ rational in the sense that the symplectic area of relative disk classes forms a discreet subset of $\bb{R}$.
\end{definition}

One has the following lemma for finding such Lagrangians:

\begin{lemma}{(Fibered Lagrangian Construction Lemma)}\cite{floerfibrations}\label{fiberedlagrangianlemma}
Let $L_F\rightarrow L\rightarrow L_B$ be a connected sub-bundle. Then $L$ is Lagrangian with respect to $\epsilon a+\pi^*\omega_B$ if and only if

\begin{enumerate}
\item\label{paralleltransinv} $L$ is invariant under parallel transport along $L_B$ and
\item\label{normalizing} there is a point $p\in L_B$ such that $a_p\vert_{TL_F\oplus H_L}=0$, where $H_L:=H_a\cap TL$ is the connection restricted to $L$.
\end{enumerate}

\end{lemma}

\section{Holomorphic sections of the pull-back} \label{holomorphicsectionssection}
We consider a $J_B$-holomorphic disk 
$$u:(D,\partial D)\rightarrow (B,L_B)$$
and review a strategy for trivializing $u^*E$ holomorphically. Then, we apply similar principles to achieve transversality for holomorphic sections.

To simplify notation, let $a:=u^*a_{h}$ be the pullback minimal coupling form. We get a symplectic connection on $u^*E$ via 
$$H:=TF^{\perp_{a}}$$
One can vary the connection as in [Chapter 8 \cite{ms2}] or \cite{salamonakveld}. Let
$$\sigma\in \Lambda^1(TD,C^\infty_0(u^*E))$$
be a one-form on $D$ with values in zero fiber-average smooth functions, and denote by $X_\sigma$ the associated fiberwise-Hamiltontian vector field valued one-form. It follows that $u\vert_{\partial D}^*L$ defines a Hamiltonian isotopy path of Lagrangians in the isotopy class of $L_F$. To see that this is actually a loop (c.f. \cite{salamonakveld}), we use the fact that $L$ is invariant under parallel transport [Lemma \ref{fiberedlagrangianlemma}].

Let
$$\mc{J}^{vert}(F):=C^\infty(D,\mc{J}(F,\omega_F))$$
Following \cite{gromov,ms2}, there is a canonical almost complex structure on the bundle $u^*E$:
\begin{lemma}[Lemma 8.2.8 \cite{ms2}]\label{uniqueacs}
For every $J_F\in \mc{J}^{vert}(F)$ and a connection $H$ as above, there is a unique almost complex structure $J_H$ on $u^*E$ so that
\begin{enumerate}
\item $J_H\vert_{TF}=J_F$
\item $\pi$ is holomorphic with respect to $(j,J_H)$, where $j$ is the standard integrable complex structure on $D$
\item $J_H$ preserves $H$.
\end{enumerate}
\end{lemma}

In trivializing coordinates and the associated trivial connection, such a structure takes the form
\begin{equation}
J_\sigma:= 
\begin{bmatrix}
J_F & J_F\circ X_\sigma-X_\sigma\circ j \\
0 & j
\end{bmatrix}
\end{equation}
and the horizontal distribution is of the form
$$(-X_{\sigma(v)},v).$$

\subsection{Moduli of holomorphic sections} \label{verticalmaslovsection}
Let $N$ be the minimal Maslov number for $L_F$. Let $L\subset u^*E$ be shorthand for the sub-bundle $u\vert_{\partial D}^*L$. For any disk or sphere class $A\in H_2(u^*E,L,\bb{Z})$ we can assign a Maslov index $\mu_F(A)$ which is the Maslov index of the path of Lagrangian subspaces $T_{u(e^{i\theta})}L$ after choosing a trivialization of $\hat{u}^*TF$ for a map $\hat{u}:D\rightarrow (F,L)$ in the class $A$.

\begin{lemma}[Lemma 6.1 \cite{salamonakveld}]\label{vmilemma}
Let $\hat{u}_i:(D,\partial D)\rightarrow (u^* E, L)$ for $i=1,2$. Then
$$\mu_F(\hat{u}_1)\equiv \mu_F(\hat{u}_2) \; \t{mod}\, N$$
\end{lemma}

\begin{definition}
The \emph{Maslov index of the pair} $(u^*E,L)$ is defined as $\mu_F(u^*E)$ mod $N$ from Lemma \ref{vmilemma}
\end{definition}

The say that the \emph{vertical Maslov index} for a lift $\hat{u}$ of $u$, equivalently for a section of $u^*E$, is $\mu_F(\hat{u})$.
$J$-holomorphic lifts can be shown to have the following transversality property:

\begin{theorem}[Theorem 5.1 \cite{salamonakveld}]\label{liftmodulithm} Let $A\in H^{section}_2(u^*E,L,\bb{Z})$ be the class of a section and $J_F\in \mc{J}^{vert}(F)$ be a section of fiberwise almost complex structures. If $\mc{M}(A,J_F,H)$ is the moduli of $J_H$-holomorphic sections in the class $A$, then there is a comeager subset of Hamiltonian connections $\mc{H}^{\t{reg}}$ so that for $H\in\mc{H}^{\t{reg}}$  we have that the moduli space is a smooth manifold with
\begin{equation}\label{liftdimensionformula}
\dim_{\bb{R}} \mc{M}(A,J_F,H)=\dim_\bb{R} L_F+\mu_F(A)
\end{equation}
\end{theorem}

\begin{remark}\label{movingboundaryproblem}
A lift of a disk $u$ is equivalent to solving the \emph{inhomogeneous moving boundary problem}:
\begin{gather*}
\hat{u}:D\rightarrow F\\
\bar{\partial}_{J,j}\hat{u}=X^{0,1}_\sigma\\
\hat{u}(e^{i\theta})\in \phi^\theta_\sigma (L_F)
\end{gather*}
where 
$$X_\sigma^{0,1}:=J_I\circ X_{\sigma\circ j}+X_{\sigma}$$
is the anti-holomorphic part of the 1-form $X_\sigma$.
\end{remark}

This remark motivates the following definition. Define the \emph{covariant derivative} of a section $\hat{u}$ with respect to a Hamiltonian connection $H_\sigma$ is
\begin{equation}
\nabla^\sigma \hat{u}:= \partial_{J_I,j}\hat{u}-X^{1,0}_\sigma
\end{equation}
and we can define the covariant energy with respect to a connection form $a_\sigma$ as

\begin{definition}[Covariant energy]
$$e_\sigma(\hat{u}):=\int_D \vert \nabla^\sigma \hat{u}\vert_{\omega_F,J_I}$$
\end{definition}

Such a quantity satisfies some topological properties. First, the \emph{curvature} of a connection $H$ is the two form

\begin{equation}\label{curvaturedef}
R_H ds\wedge dt := a_H(\partial^\sharp_s,\partial^\sharp_t) ds\wedge dt
\end{equation}
where $\partial_s^\sharp$ is the horizontal lift of $\partial_s$ to the connection $H$ (resp. $\partial_t$). We remark that by the curvature identity \ref{curvatureidentity}, the function in \ref{curvaturedef} is the zero-average Hamiltonian corresponding to the vector field $[\partial_s^\sharp,\partial_t^\sharp]^{\t{vert}}$.

\begin{lemma}\label{covariantenergyprop}[Lemma 5.2 \cite{salamonakveld}]
$$e_\sigma(\hat{u})=\int_D \hat{u}^*a_H +\int_D R_H\circ \hat{u}\, ds\wedge dt$$
\end{lemma}

\subsection{Existence of a flat connection}\label{flatconnectionsection}
We begin to specialize to the case $E=\bb{P}(\mc{V})$ for a vector bundle $V\rightarrow \mc{V}\rightarrow B$.

Let $G_\bb{C}$ denote the complexified structure group of $E$, and $\t{Hol}(D,G_\bb{C},G)$ the group of holomorphic sections of the trivial principal $G_\bb{C}$-bundle that are $G$-valued on the boundary. We use a result of Donaldson \cite{sdh} to show that there is a section which trivializes the connection on $u^* \mc{V}$.

\begin{theorem}[Theorem 1 \cite{sdh}]\label{heatflowthm}
Let $\mc{V}$ be a holomorphic vector bundle over a K\"ahler manifold $Z$. For any metric $f$ on the restriction of $\mc{V}$ to $\partial Z$, there is a unique Hermitian metric $h$ that satisfies the Yang-Mills equation:
\begin{enumerate}
\item $h\vert_{\partial Z}=f$

\item $i\Lambda F_h=0$ \label{ymequation}
\end{enumerate}
where $i\Lambda F_h$ is the $(1,1)$ part of the curvature of $h$ in the K\"ahler decomposition.
\end{theorem}

When $Z$ is a disk, \ref{heatflowthm} says 
$$\mc{V}\cong (D\times V,j\times J_V)$$
as holomorphic vector bundles due to \ref{ymequation} simply stating that $h$ induces a flat connection. 

Let $h_\infty$ denote the flat metric on $u^*\mc{V}$ from \ref{heatflowthm} that agrees with $h_0$ over $\partial D$. Then there exists a smooth section 
$$\mc{G}:(D,\partial D)\rightarrow (G_\bb{C},G)$$ such that $$h_\infty(\cdot,\cdot)=h_0(d\mc{G}\cdot,d\mc{G}\cdot)=:\mc{G}^*h_0.$$ The discussion leads to the following corollary of \ref{heatflowthm}:

\begin{corollary}\label{heatflowgaugever}
Let $u^*\mc{V}$ be a holomorphic vector bundle equipped with a fiberwise $G_\bb{C}$ invariant complex structure $J_V$ and K\"ahler metric $h_0$. Then there is complex gauge transformation $\mc{G}\in \mc{C}^\infty (D,G_\bb{C},G)$ that is $G$ valued on $\partial D$ so that $h_\infty:=\mc{G}^*h_0$ is flat.
\end{corollary}

Next, we show that we can obtain holomorphic sections of the projectivization by looking at the gauge-transformed pullback. For $\mc{G}\in \mc{C}^\infty(D,G_\bb{C},G)$ we get an isomorphism of bundles $$\mc{G}:u^*\bb{P}(\mc{V})\rightarrow u^*\bb{P}(\mc{V})$$
which has a well-defined tangent map $d\mc{G}$. 
\begin{definition}
For the pullback symplectic connection $H$ discussed in the beginning of section \ref{holomorphicsectionssection}, define the connection 
$$\mc{G}_*H:=d\mc{G}(H)$$ on $u^*\bb{P}(\mc{V})$ with respect to the fiberwise symplectic form $\mc{G}^*\mc{I}m (h)$.
\end{definition}
On $\mc{G}u^*\bb{P}(\mc{V})$ we define the almost complex structure
\begin{equation}
\mc{G}^*J_H:=d\mc{G}\circ J_H\circ d\mc{G}^{-1}
\end{equation}
\begin{lemma}\label{hologaugelemma}
$\mc{G}^*J_H$ is the unique almost complex structure compatible with $\mc{G}_*H$ from Lemma \ref{uniqueacs}. Hence, $\mc{G}$ preserves holomorphic sections.
\end{lemma}
\begin{proof}
$\mc{G}_*H$ is of the form

$$(-\mc{G}_*X_y, y)$$

Thus, the unique almost complex structure on $\mc{G}(v^*E)$ from Lemma \ref{uniqueacs} is
\begin{equation}\label{gaugeacs}
\begin{bmatrix}
J_I  & J_I \mc{G}_*X_\sigma-\mc{G}_*X_{\sigma\circ j}\\
0 & j
\end{bmatrix}
\end{equation}
in the trivial connection, which is equal to $\mc{G}^*J_H$ by the $G_\bb{C}$ equivariance of $J_I$ and the fact that $\mc{G}$ acts as the identity of on the $TD$ component of the trivial splitting.

To check the second statement, let $v$ be a $J_H$-holomorphic section $\pi_F$ be the projection onto the fiber in the above trivialization and $dv^F:=\pi_{TF}\circ dv$ be the projection onto the $TF$ component in the induced trivial connection. Since $v$ is $J_H$ holomorphic, we have
\begin{equation}\label{hamiltonianperturbedeqn}
\bar{\partial}_{J_I,j}\pi_F v=dv^F+J_I\circ dv^F\circ j=X_\sigma^{0,1}
\end{equation}
wherefore $\mc{G}v$ satisfies
\begin{align*}\bar{\partial}_{J_I,j}\pi_F\mc{G}v &=\mc{G}_*dv^F+J_I\circ\mc{G}_*dv^F\circ j
\\ &=\mc{G}_*\bar{\partial}_{J_I,j}\pi_Fv=\mc{G}_*X_\sigma^{0,1}
\\ &=(\mc{G}_*X_\sigma)^{0,1}
\end{align*}
again by the $G_\bb{C}$-invariance of $J_I$. By remark \ref{movingboundaryproblem} a solution to \ref{hamiltonianperturbedeqn} is equivalent to a $J_H$-holomorphic section, so we are done.
\end{proof}

\subsection{Holomorphic lifts to $(E,L)$}\label{holomorphicliftsubsection}
Continuing from the last subsection, we show that there is at least one holomorphic lift of a base disk through each fiber point. However, we cannot always achieve transversality for a given lift by gauge transformation: In example \ref{realprojectivebundles} the structure group is $S^1$, and in lieu of Theorem \ref{transversalitythm} such an action does not provide a large enough space of connections to achieve transversality, particularly for the constant sections that are fixed by the gauge group. We get around this by allowing transformations of $u^*E$ that take values in $PU_2$ over the interior of the disk. Such a group acts via holomorphic symplectomorphisms and the differential $\mf{g}\rightarrow T_pF$ is surjective at every point, allowing sufficiently many perturbations to achieve transversality. In this section, we record the existence of a lift and prove a transversality result.

Since $E=\bb{P}(\mc{V})$ is the projectivization of a vector bundle and the symplectic connection is induced by a choice of Hermitian metric on $\mc{V}$, we have canonical action via $PU_{m+1}$ on $E$ that preserves the fiberwise symplectic form and the standard complex structure. Fix a (left invariant) inner product $\langle\cdot,\cdot\rangle$ on $PU_{m+1}$ with norm $\vert\cdot\vert$. The left action of such a group on $v^*E$ preserves the fiberwise symplectic form and integrable complex structure and acts on connection by the discussion in subsection \ref{flatconnectionsection}. We define the space of perturbations
$$\mf{G}^l:= W^{l,2}_{\partial D, D}(G,PU_{m+1}):=\bigg\lbrace \mc{G}:(\partial D,D)\rightarrow (G,PU_{m+1}):\int_D \vert D^i \mc{G} \vert^2 dxdy <\infty \bigg\rbrace$$
as the Banach manifold of $l$-differentiable $2$-integrable $PU_{m+1}$-valued gauge transformations of $u^*E$ that take values in $G$ over $\partial D$. Take $l$ large enough so that such a space includes into continuous functions by Sobolev embedding. At a given point the tangent space is written as
$$T_\mc{G}\mf{G}^l= W^{l,2}_{\partial D,D}(\mf{g},\mf{pu}_{m+1})$$
With local charts given by the exponential map and transitions by left shift. Note that such a space can also be written as the space of bundle isomorphisms covering the identity that are differentiable to some degree.

\subsubsection{Stabilizing Divisors}
A crucial component of the perturbation scheme is the existence of a stabilizing divisor for a certain class of rational Lagrangians $L_B$.
\begin{definition}
We say a Lagrangian $L_B\subset B$ is \emph{strongly rational} if there exists a line bundle-with-connection on $B$ with a section $\chi$ whose restriction $\chi\vert_{L_B}$ is covariant constant.
\end{definition}

\begin{definition}
We say that an almost complex submanifold $D_B\subset B\setminus L_B$ with $[D_B]^{PD}=N[\omega_B]$ is a \emph{stabilizing divisor} for $L_B$ if $\forall [v]\in H^\circ_2(B,L_B,\bb{Z})$ with positive symplectic area we have 
$$v^{-1}( D_B) \neq \emptyset$$
\end{definition}

For a stabilizing divisor $D_B$, call an almost complex structure $J_B$ \emph{adapted to $D_B$} if it is an almost complex submanifold with respect to $J_B$. The following is made possible by a number of results, with the stabilizing property made possible by the work in \cite{CW1}.

\begin{theorem}[\cite{CW1} Section 3.1,\cite{agm},\cite{SD}]
For $B,L_B$ rational, there exists a stabilizing divisor with an adapted almost complex structure.
\end{theorem}

Let $D$ be a lift of the divisor $D_B$ to the total space. Since $D_B$ is stabilizing for $L_B$, we have that every representative of a non-trivial $[u]\in H_2(B,L_B,\bb{Z})$ intersects $D_B$ at least once. It follows that every (section or multisection) of $u^*E$ intersects $D$ in at least one point. The expected dimension of a configuration in $u^*E$ through a point $p\in L_F$ with prescribed intersection to $D_E$ in the relative class $[A]$ is
\begin{equation}\label{liftdimensionformula2}
\I (A)= \sum_{i=1}^j I(v_i)-2\langle A, D_E\rangle_{u^*E}+2\vert \text{Edge}^\bullet(\gamma_u)\vert
\end{equation}
where  $\text{Edge}^\bullet(\gamma_u)$ is the number of interior marked points on $u$ prescribed to the divisor $D_B$. By positivity of intersection of holomorphic curves with the divisor, $\langle A, D_E\rangle_{v^*E}$ is equal to the sum of the degrees of vanishing of the normal derivative in coordinates adapted to the divisor.
On the other hand, we will usually assume that one can perturb the almost complex structure and/or divisor so that each disk intersects the divisor transversely, making the last two terms on the right hand side of \ref{liftdimensionformula2} irrelevant.

\begin{theorem}\label{liftthm}

Let $A\in H_2( u^*E,u^*L, \bb{Z})$ be a relative section class and $p\in L$ a generic point. There is a comeager set of smooth gauge transformations $\mf{G}^{\infty,reg}\subset \mf{G}^{\infty}$ such that the moduli space
$$\mc{M}_{j,\mc{G}^*J_H}(u^*E,u^*L,A,p)$$
of $\mc{G}^*J_H$-holomorphic sections for $\mc{G}\in \mf{G}^{\infty,reg}$ in the class $A$ is a smooth manifold of expected dimension
$$\dim \mc{M}_{j,\mc{G}^*J_H}(u^*E,u^*L,A,p)=\I (A)$$

Moreover, there exists a trivialization of $u^*E$ that preserves the fiberwise complex structure and the boundary Lagrangian. In particular, let $p\in L_F\cap \pi^{-1}(1)$. There exists a (regular) $\mc{G}^*J_H$-holomorphic section $\hat{u}$ of $u^*(E,L)$ with $\hat{u}(1)=p$ and $\mc{G}\in \mf{G}^{\infty,reg}$
\end{theorem}
The final statement was initially made in \cite{fiberedpotential}. We include a proof for completeness and clarity.
\begin{remark}\label{computeremark}
The proof of Lemma \ref{hologaugelemma} is exactly the same in case $\mc{G}$ takes values in $PU_{m+1}$ over $D$. 
Hence by Theorem \ref{liftthm} it suffices to compute holomorphic sections through $p$ in the given connection and then perturb with a generic gauge transformation that is close enough to the identity so as to not change the relative class.
\end{remark}
\begin{proof}
First we show the existence of such a section, and then prove regularity for more general maps.
Let $\mc{G}_\bb{C}$ denote the flattening gauge transformation from Corollary \ref{heatflowgaugever} The action by $G$ is Hamiltonian, so it follows that $\mc{G}_\bb{C}(L)$ is fiberwise Lagrangian. The parallel transport along $\mc{G}_\bb{C}(H)\vert_{\partial D}$ is given by

$$\mc{G}_\bb{C}\circ \nabla^H\circ \mc{G}_\bb{C}^{-1}$$

where $\nabla^H$ is transit along $H\vert_{\partial D}$. Since $L$ is invariant under parallel transit, it follows that $\mc{G}_\bb{C}(L)$ must also be. Hence the fibered Lagrangian construction Lemma \ref{fiberedlagrangianlemma} tells us that $\mc{G}_\bb{C}(L)$ is Lagrangian. 

Since $\mc{G}_{\bb{C}*} H$ is flat, parallel transport starting at $u(1)$ defines a holomorphic trivialization
$$\mc{T}:[\mc{G}_\bb{C}u^*E,\mc{G}_{\bb{C}*}L,J_H]\cong [D \times F,\partial D\times L_F, j\times J_I]$$
Then by lemma \ref{hologaugelemma} the trivialization
$$\mc{T}\circ \mc{G}_\bb{C}:(u^*E,J_H)\rightarrow (D\times \bb{P}^m,j\times J_I)$$
is a biholomorphism. Let $p\in L_F$, and define $$v_p^0(z):= \mc{G}_{\bb{C}}^{-1}\circ \mc{T}^{-1}(z,p)$$

The map $\mc{G}_\bb{C}^{-1}\circ\mc{T}^{-1}$ is given precisely by left multiplication on the total pullback via
$$\mc{G}_{\mc{T}}: (D,\partial D)\rightarrow (G_\bb{C},G)$$
as in theorem \ref{heatflowgaugever}.
Moreover, such a disk is holomorphic by Lemma \ref{hologaugelemma}. Hence, we have shown the existence of a section $v^0_p$.

Next we check the surjectivity of the linearized operator in the case of no prescribed tangecies. The proof is similar to the proof of Theorem \ref{transversalitythm} (see \cite{floerfibrations}), except that we restrict ourselves to $PU(m+1)$ connections.  Choose a metric $(g,\nabla_g)$ with Levi-Civita connection on $u^*E$ such that $L$ is totally geodesic, pick $\frac{k}{m}>\frac{1}{2}$ so that Sobolev embedding into continuous functions holds, and consider the Banach manifold
$$\Omega^{k,2}_{A, sec}(D,\partial D, u^*E,L)$$
of sections in a relative class $A$ that are $k$ times weakly differentiable with square integrable derivatives in regards to $\nabla_g$.  Given a map $v$ in such a space, we get a local chart
$$W_v\subset W^{k,2}(D,\partial D,v^*TF, v^*TL_F)$$
that we can map into $\Omega^{k,2}$ by geodesic exponentiation according to $\nabla_g$. The transition map between two such charts is given by parallel transport of vector fields between the chart centers. Set
$$\mc{B}_{A}^k:=\Omega^{k,2}_{A,sec}\times \mf{G}^l$$
Over such a manifold we have a Banach vector bundle $\mc{E}^{k-1}$ whose fiber above a map $(v,\mc{G})$ is
$$(\mc{E}^{k-1})_{v,\mc{G}}:= \Lambda_{\mc{G}^*J,j}^{0,1}(D,v^*TF)_{k-1,2}$$
and whose transition functions are given expontiation in $PU_{m+1}$ and by parallel transport with regards to a Hermitian connection $\nabla^{\mc{G}}$ on $u^*TE$ that preserves the almost complex structure $\mc{G}^*J$ (see [3.2 \cite{ms2}]).

We have a section
\begin{align*}
\partial(v,\mc{G}):=&\bar{\partial}_{\mc{G}^*J,j}v=dv+\mc{G}^*J\circ dv\circ j
\end{align*}
with linearization at a holomorphic $v$ given by
\begin{equation}D_{v,\mc{G}^*J}\bar{\partial}:=D_{v,\mc{G}^*J}+J_\mf{h}\circ dv\circ j
\end{equation}

Here $D_{v,\mc{G}^*J}$ the Cauchy-Riemann (hence Fredholm) linearization of $\bar{\partial}_{\mc{G}^*J,j}$ with respect to the variable $v$, $\mf{h}$ is a section of projective skew-Hermitian matrices in the Lie algebra $T_{I}\mf{G}^l$ with values in $TG$ over $\partial D$, and $J_\mf{h}$ is the derivative of \ref{gaugeacs} with respect to $\mc{G}$.

We define the \emph{universal moduli space} as 
$$\mc{M}_{sec}^{k,2}(u^*E,L,A):=\bar{\partial}^{-1}(\mc{B})$$
where $\mc{B}$ is identified with the zero section. The first goal will be to show that $D_{v,\mc{G}^*J}\bar{\partial}$ is surjective when $v$ is $\mc{G}^*J$-holomorphic to prove that the universal space is a smooth manifold. Lastly, we apply a Sard-Smale argument to find a set of comeager $\mc{G}$ such that the projection
$$\pi_{\mf{G}}:\mc{M}_{sec}^{k,2}(u^*E,L,A)\rightarrow \mf{G}^l$$
has a comeager set of regular values.

Suppose $D_{v,\mc{G}^*J}\bar{\partial}$ is not surjective. Since it is a Fredholm operator, we can assume that its image is closed. By the Hahn-Banach theorem there is a non-zero section
$$\eta\in \Lambda^{0,1}_{\mc{G}^*J,j}(D,v^*TF)$$
such that
\begin{align}&\label{derivequation}\int_D\langle D_{v,\mc{G}^*J}\xi,\eta\rangle=0 
\\& \label{perturbequation}\int_D\langle J_\mf{h}\circ dv\circ j,\eta\rangle=0
\end{align}
for all $\xi\in W^{k,2}(D,\partial,v^*TF,v^*TL_F)$ and all $\mf{h}\in T_{I}\mf{G}^l$.

We have that $\eta$ is in the kernel of the adjoint real Cauchy-Riemann operator
$$D_{v,\mc{G}^*J}^*\eta=0.$$
Choose a point $p$ where $\eta\neq 0$ and $p\in U$, which is possible since $\eta$ cannot be zero on a dense set.

The map $\mf{pu}_{m+1}\rightarrow TF$ is surjective, so choose a skew-Hermitian matrix $g$ so that 
$$\langle J_g\circ du(p)\circ j,\eta(p)\rangle >0$$
and extend $g$ to a section of skew-Hermitian matrices $\mf{h}$ by a bump function so that
$$\langle J_\mf{h}\circ du\circ j,\eta\rangle \geq 0$$
Thus, \ref{perturbequation} is strictly positive for the supposed $\eta$, which is a contradiction.

The case for surjectivity in the presence of presribed tangentcies to $\pi^{-1}(D_B)$ is covered comprehensively in [\cite{CM} Lemma 6.6] and we refer the reader to such.

By the implicit function theorem, $\mc{M}_{sec}^{k,2}(u^*E,L,A)$ is a class $C^q$ manifold for $q<l-k$.

By the Sard-Smale theorem, and for $k,l>>1$, the set of regular values of the projection $\pi_\mf{G}$ is comeager. Since the kernel and cokernel of $\pi_\mf{G}$ are isomorphic to that of $D(v,\mc{G})$, we have that 
$$\pi^{-1}(\mc{G})=\mc{M}_{j,\mc{G}^*J_H}(u^*E,u^*L,A,p)$$
is a $C^?$ manifold of dimension $\I (A)$. The existence of a comeager set for smooth perturbation follows from an argument due to Taubes (see proof \cite{ms2} Theorem 3.1.5). The fact that all elements of the moduli space are smooth solutions follows from an elliptic bootstrapping argument. Thus, each moduli space $\mc{M}_{j,\mc{G}^*J_H}(u^*E,u^*L,A,p)$ is smooth of expected dimension.
\end{proof}


\section{Transversality/Compactness}\label{transection}

The fact that transversality for full configurations of expected dimension $\leq 1$ follows directly from the author's work in \cite{fiberedpotential}\cite{floerfibrations}. We provide a brief recall of the main notions and theorem in this section. Essentially, holomorphic configurations \emph{(treed disks)} are based on trees with each vertices corresponding to holomorphic disks/spheres and edges to Morse flow lines/nodes/marked points mapping to the divisor. The notions of a treed disk and the notation are based on work in \cite{CW2} and initially inspired by \cite{birancorneapearl}
\begin{figure}[h]
\includegraphics[scale=1]{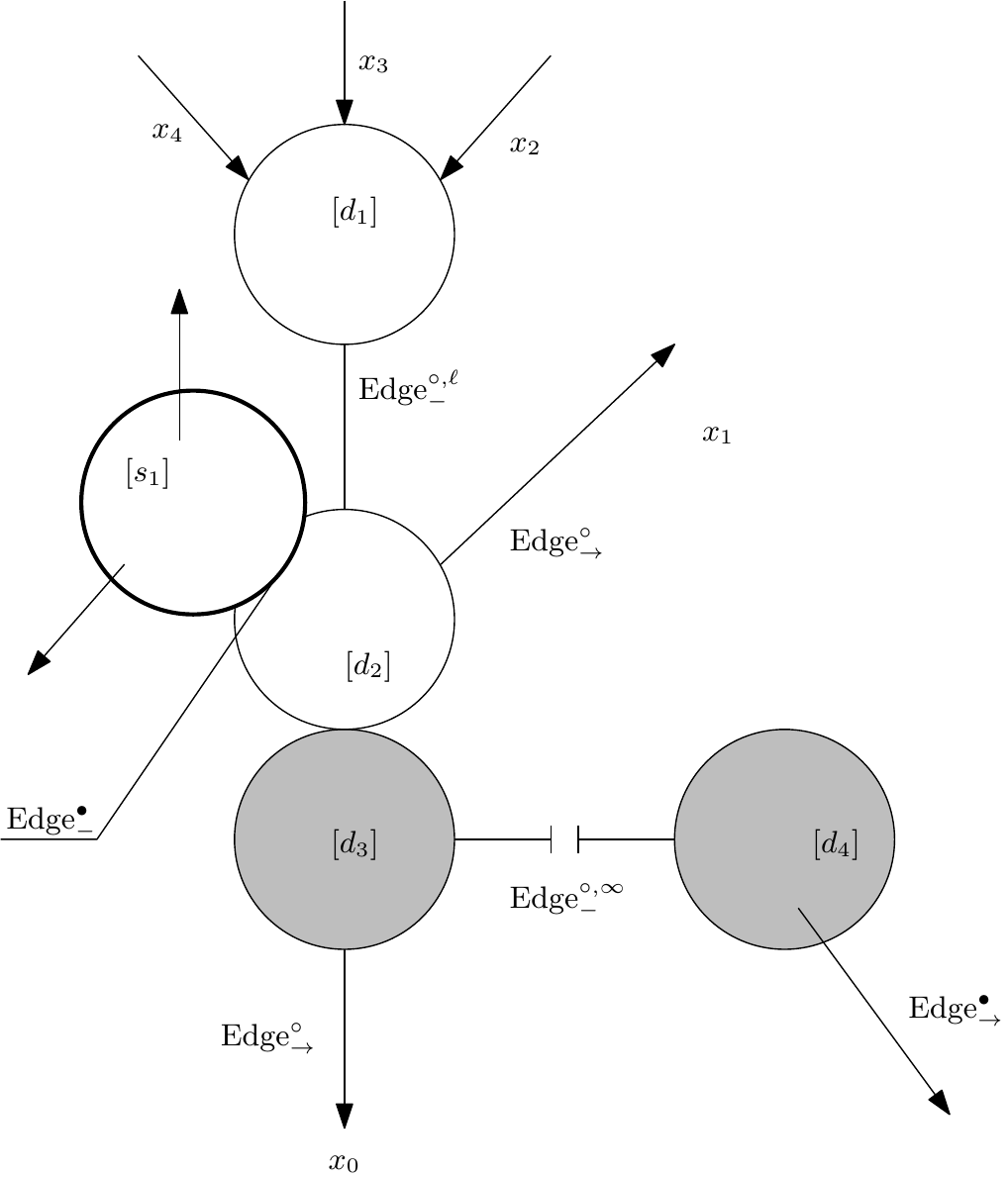}
\caption{\cite{floerfibrations} The geometric realization of a treed disk. Shading correspond to domains mapped to constants under $\pi$.}
\end{figure}
Let $\g$ be metric tree with the additional labeling of edges and vertices:
\begin{center}
\begin{tabular}{||c | c ||}
\hline
$\v_s(\g)$ & spherical vertices\\
\hline
$\v_d(\g)$ & disk vertices\\
\hline
$\e_\rightarrow^\bullet (\g)$ & interior markings ($\sim$ interior special points)\\
\hline
$\e_-^\bullet(\g)$ & interior nodes ($\sim$ interior special points)\\
\hline
$\e_\rightarrow^\circ (\g)$ & boundary markings ($\sim$ boundary special points)\\
\hline
$\e_-^\circ(\g)$ & boundary nodes ($\sim$ boundary special points)\\
\hline
$\ell:\e_-^\circ (\g)\rightarrow [0,\infty]$& boundary node length\\
\hline
$m:\e_\rightarrow^\bullet(\g)\rightarrow \bb{Z}_{\geq 0}$ & divisor intersection multiplicity\\
\hline
$[\cdot]_d:\v_d(\g)\rightarrow H^d_2(E,L)$ & disk classes\\
\hline
$[\cdot]_s:\v_s(\g)\rightarrow H^s_2(E)$ & sphere classes\\
\hline
\end{tabular}
\end{center}

We require each disk/sphere domain with $\pi_*[d]\neq 0$ to be stable in the sense that the number of boundary special points plus twice the number of interior special points is $\geq 3$ (\emph{$\pi$-stability}).
On spheres the perturbation in Theorem \ref{liftthm} becomes more complex, so we resort to using a more general perturbation of the connection as in [Chapter 8 \cite{ms2}]. Following [Section 5.3 \cite{floerfibrations}], let 
$$\ham (F,B):=\Lambda^1(B,C^\infty_0(F,\bb{R}))$$
denote the space of one forms on $B$ with values in fiberwise zero average smooth functions. For $\sigma\in \ham(F,B)$, one obtains a different connection form on the total space by considering
$$a_\sigma:=a-d\sigma$$
that we denote by $H_\sigma$. The holonomy of $H_\sigma$ differs infinitesimally from that of $H$ by the fiberwise Hamiltonian vector field of $\sigma$, which we denote by $X_\sigma$.
Let 
\begin{equation*}\mc{J}^l(E,\omega_{\epsilon,a_\sigma}):= \bigg\lbrace J \text{ from Lemma \ref{uniqueacs}},\, J\vert_F=J_I,\, J \text{ is tamed by }\epsilon a_\sigma+\pi^*\omega_B\bigg\rbrace .
\end{equation*}
be the subset of tamed structures adapted to the connection. This can be realized as an open subset of $\mc{J}^l(B,\omega_B)$ and thus is a Banach manifold.

Let $\mc{U}_\g$ denote the \emph{universal treed disk} of type $\g$; that is, we have a fiber bundle
$$\mc{U}_\g\rightarrow \mf{M}_\g$$
whose fibers are isomorphism classes of treed disks and $\mf{M}_\g$ is the moduli space of treed disks of type $\g$. More generally than in subsection \ref{holomorphicliftsubsection}, let
$$\mf{G}^l(\mc{U}_\g):= W^{l,2}_{\mc{U}_\g}(PU_{m+1}):=\bigg\lbrace \mc{G}:\mc{U}_\g\rightarrow PU_{m+1}:\int_D \vert D^i \mc{G} \vert^2 dxdy <\infty \bigg\rbrace$$
be the Banach space of $l$ differentiable square integrable functions into $PU_{m+1}$.

Choose neighborhoods $\mc{S}_\g^o$ resp. $\mc{T}_\g^o$ of each surface part $\mc{S}_\g$ resp. edge part $\mc{T}_\g$ not containing the special points and boundary on each surface component resp. not containing $\infty$ on each edge.

\begin{definition}\cite{floerfibrations}\label{perturbationdatadef} Fix an adapted $J_{D_B}$ to $D_B$. Let $(\g,\underline{x})$ be a $\pi$-stable type of treed disk with a Morse labeling. A class $C^l$ \emph{fibered perturbation datum} for the type $\g$ is an element
$$\mc{G}_\g\in \mf{G}^l(\mc{U}_\g) $$
together with piecewise $C^l$ maps 
$$(J_\g,\sigma):\mathcal{U}_\g\rightarrow \mathcal{J}^l(E,\omega_{H,K})\times \ham (F,B) $$
$$X:\mc{T}_\g\rightarrow \t{Vect}^l(TL_F)\oplus \t{Vect}^l(TL_B)$$
denoted $P_\g$, that satisfy the following properties:
\begin{enumerate}
\item $\sigma\equiv 0$ outside sphere components and on the neighborhoods $\mc{S}_\g - \mc{S}_\g^o$

\item $J_B\equiv J_{D_B}$ on the neighborhoods $\mc{S}_\g - \mc{S}_\g^o$,
\item $\mc{G}_\g$ takes values in $G$ on edges, the boundary of each disk component, $\mc{S}_\g - \mc{S}_\g^o$, and vertices with $\pi_*[v]=0$, and vanishes on spheres, neighborhoods of interior special points, and the neighborhoods $\mc{T}_\g-\mc{T}_\g^o$

\item $X$ is identified with a map $X:\mc{T}_\g\rightarrow \t{Vect}^l(TL_F)\oplus \t{Vect}^l(H_L)$ via horizontal lift and $X\equiv X_f$ in the neighborhood $\mc{T}_\g-\mc{T}_\g^o$ of $\infty$.
\end{enumerate}

Let $\mc{P}^l_\g(E,D)$ denote the Banach manifold of all class $C^l$ fibered perturbation data (for a fixed $J_{D_B}$).
\end{definition}

\begin{definition}
A perturbation datum for a collection of combinatorial types $\gamma$ is a family $\underline{P}:=(P_\g)_{\g\in\gamma}$
\end{definition} 

In order for compactness to hold, such a collection must satisfy some coherence axioms, which we reference

\begin{definition}{[Definition 2.11 \cite{CW2}]} \label{coherentdefinition} A fibered perturbation datum $\underline{P}=(P_\Gamma)_{\g\in\gamma}$ is \emph{coherent} if it is compatible with the morphisms on the moduli space of treed disks in the sense that
\begin{enumerate}
\item\emph{(Cutting edges axiom)} if $\Pi:\g'\rightarrow \g$ cuts an edge of infinite length, then $P_{\g}=\Pi_* P_{\g'}$,
\item\emph{(Collapsing edges/making an edge finite or non-zero axiom)} if $\Pi:\g\rightarrow \g'$ collapses an edge or makes an edge finite/non-zero, then $P_{\g}=\Pi^*P_{\g'}$,
\item\emph{(Product axiom)} if $\g$ is the union of types $\g_1, \g_2$ obtained from cutting an edge of $\g$, then $P_\g$ is obtained from $P_{\g_1}$ and $P_{\g_2}$ as follows: Let $\pi_k:\mf{M}_\g\cong \mf{M}_{\g_1}\times\mf{M}_{\g_2}\rightarrow\mf{M}_{\g_k}$ 
denote the projection onto the $k^{th}$ factor, so that $\mc{U}_\g$ is the unions of $\pi_{1}^{*}\mathcal{U}_{\g_1}$ and $\pi_{2}^{*}\mathcal{U}_{\g_2}$. Then we require that $P_\g$ is equal to the pullback of $P_{\g_k}$ on $\pi_{k}^*\mathcal{U}_{\g_k}$
\item \emph{(Ghost-marking independence)} If $\Pi:\g'\rightarrow \g$ forgets a marking on components corresponding to vertices with $[v]=0$ and stabilizes, then $P_{\g'}=\Pi^* P_{\g}$.
\end{enumerate}
\end{definition}
We refer the reader to section 2 of \cite{CW2} for the precise definition of the moduli of treed disks and the morphisms between them.
Let $u:C_\g\rightarrow E$ be a Floer trajectory class based on a $\pi$-stable combinatorial type $\g$. $u$ is called \emph{adapted} to $D$ if
\begin{enumerate}
\item(Stable domain) The geometric realization of $\g$ after forgetting vertices with $\pi_*[v]=0$ is a stable domain;
\item(Non-constant spheres)\label{spheresaxiom} Each component of $C$ that maps entirely to $D$ is constant;
\item(Markings)\label{markingsaxiom} Each interior marking $z_i$ maps to $D$ and each component of $u^{-1}(D)$ contains an interior marking.
\end{enumerate}
Denote the set of adapted $P_\g$-Floer trajectories by 
$$\mc{M}_\g(\underline{x},P_\g,D)$$
Finally, we have the expected dimension of a given moduli space
\begin{align}\label{indexformulatotalspace}
\I (\g,\underline{x}):&= \mathrm{dim}W^+_{X}(x_0)-\sum_{i=1}^{n}\mathrm{dim}W^+_{X}(x_i) + \sum_{i=1}^{m} I(u_i)+n-2 -|\e^{\circ,0}_-(\g)|\\&-|\e_-^{\circ,\infty}(\g)|-2|\e^\bullet_-(\g)| -|\e^\bullet_\rightarrow(\g)|-2\sum_{e\in \e^{\bullet}_\rightarrow} m(e)
\end{align}

In the following, by an uncrowded type we a type of treed disk whose constant components have at most one interior marking.
\begin{theorem}[Transversality, Theorem 6.1 \cite{floerfibrations}] \label{transversalitythm}
Let $E$ be a symplectic K\"ahler fibration, $L$ a fibered Lagrangian. Suppose that we have a finite collection of uncrowded, $\pi$-adapted, and possibly broken types $\lbrace\g\rbrace$ with $$\I(\g,\underline{x})\leq 1.$$ 
Then there is a comeager subset of smooth regular data for each type $$\mc{P}^{\infty,reg}_\g(E,D)\subset\mc{P}^\infty_\g$$
and a selection $$(P_\g)\in \Pi_\g \mc{P}^{\infty,reg}_\g$$ that forms a regular, coherent datum. Moreover, we have the following results about tubular neighborhoods and orientations:

\begin{enumerate}
\item\label{tubularneighborhoods}\emph{(Gluing)} If $\Pi:\g\rightarrow\g'$ collapses an edge or makes an edge finite/non-zero, then there is an embedding of a tubular neighborhood of $\mc{M}_\g (P_\g,D)$ into $\overline{\mc{M}}_{\g'}(P_{\g'},D)$, and
\item \emph{(Orientations)} if $\Pi:\g\rightarrow\g'$ is as in \ref{tubularneighborhoods}, then the inclusion $\mathcal{M}_{\g} (P_{\g},D)\rightarrow \overline{\mathcal{M}}_{\g'} (P_\g,D)$ gives an orientation on $\mc{M}_\g$ by choosing an orientation for $\mc{M}_{\g'}$ and the outward normal direction on the boundary. 
\end{enumerate}
\end{theorem}
The proof is almost exactly the same as that of [Theorem 6.1 \cite{floerfibrations}] after considering the ideas from Theorem \ref{liftthm}. The only difference is the use of $PU_{m+1}$-connections on disk components that match with $G$-connections on the boundary instead of more general Hamiltonian connections.

Finally, the coherence axioms combined with transversality provide a compactness result:
\begin{theorem}[Compactness, Theorem 7.1 \cite{floerfibrations}]\label{compactnessthm}
Let $\g$ be an uncrowded type with $\I(\g,\underline{x})\leq 1$ and let $\mc{P}=(P_\gamma)$ be a collection of coherent, regular, stabilized fibered perturbation data that contains data for all types $\gamma$ with $\I(\gamma,\underline{x})\geq 0$ and from which $\g$ can be obtained by (making an edge finite/non-zero) or (contracting an edge). Then the compactified moduli space $\overline{\mc{M}}_\g(D, P_\g)$ contains only regular configurations $\gamma$ with broken edges and unmarked disk vertices. In particular, there is no sphere bubbling or vertical disk bubbling. 
\end{theorem}

\section{The Fukaya Algebra}\label{fukayaalgebrasection}
We define the Fukaya algebra arising from counting isolated holomorphic configurations with a given number of inputs. In \cite{fiberedpotential} we show that a similar algebra satisfies the axioms of a $A_\infty$-algebra.

\subsubsection{Grading}
To begin, we have a short review of grading. Let $\lag(V,\omega)$ be the Lagrangian Grassmannian associated to a symplectic vector space and let $N$ be an even integer. The coverings of $\lag$ with deck transformation group $\bb{Z}_N$ are classified by classes in $H^1(\lag,\bb{Z}_N)$. Specifically, there is a well-defined \emph{Maslov class} $\mu_V\in H^1(\lag(V,\omega),\bb{Z})$, so let $\lag^N(V,\omega)$ be the covering space corresponding to the image of $\mu_V$ in $H^1(\lag,\bb{Z}_N)$.

Let $\mc{L}(B)\rightarrow B$ be the fiber bundle with fiber $\mc{L}(B)_b=\t{Lag}(T_b B,\omega)$ the Lagrangian subspaces in $T_b B$. Following Seidel \cite{seidelgraded} an $N$-fold \emph{Maslov covering} of $B$ is the fiber bundle
$$\mc{L}(B)^N\rightarrow E$$
with fibers $\lag^N(T_b B,\omega)$. For an orientable Lagrangian $L_B$, there is a canonical section $s:L_B\rightarrow \mc{L}(E)\vert_L$. A $\bb{Z}_N$-\emph{grading} of $L_B$ is a lift of $s$ to a section 
$$s^N:L\rightarrow \mc{L}^N(B)\vert_{L_B}.$$
It is natural to ask about the existence of such a grading, for which we have the following answer:
\begin{lemma}[Lemma 2.2 \cite{seidelgraded}]
$B$ admits an $N$-fold Maslov cover iff $2c_1(B)\in H^2(E,\bb{Z}_N)$ is zero. 
\end{lemma}

Let $\mu^N\in H^1(\mc{L}(B),\bb{Z}_N)$ the global Maslov class mod $N$.
\begin{lemma}[Lemma 2.3 \cite{seidelgraded}]
$L_B$ admits a $\bb{Z}_N$-grading iff $0=s^*\mu^N\in H^1(L_B,\bb{Z}_N)$.
\end{lemma}

Henceforth, we fix an $N$-fold grading $s^N$ on $L_B$. This gives rise to a \emph{forgetful grading} $\pi^*s^N:L\rightarrow \mc{L}^N(B)$ on $L$, where we will keep track of the vertical grading with the vertical Maslov index as in section \ref{verticalmaslovsection}.

\subsection{The Floer complex}\label{floercomplexsubsection}

 We will use a variant of the Novikov ring in two variables. Choose a regular coherent perturbation datum $\mc{P}:=\lbrace P_\g\rbrace_\g$ as in Theorem \ref{transversalitythm}. Let $H_2^\circ (B,L_B,\bb{Z})$ denote the cone 
of relative homology classes that have a representative as a union of $J$-holomorphic configurations for some collection of tamed $J$'s. Let $H_2^{\circ,\pi} (E,L,\bb{Z})$ denote the same type of cone but only for tamed almost complex structures that make $\pi$ holomorphic. Note that $$\pi_* H_2^{\circ,\pi} (E,L,\bb{Z})\subset H_2^\circ (B,L_B,\bb{Z})$$ Hence, by Theorem \ref{liftthm} we have
$$\pi_*H^{\circ,\pi}_2(E,L,\bb{Z})= H^\circ_2 (B,L_B,\bb{Z})$$
Let $q$ be a formal variable and define 
$$\mc{Q}(B,L_B):=\lbrace 0,q^A:A\in H_2^\circ (B,L_B,\bb{Z})\vert\, q^A\cdot q^B=q^{A+B}, q^0=1\rbrace$$
denote the ring generated by formal symbols. We have a morphism under the monoidal structure
\begin{gather*}\deg (\cdot):\mc{Q}\rightarrow \bb{R}_{\geq 0}\\
\deg( q^A ):=\omega_B(A):=\int_{\mc{C}}\omega_B^* u
\end{gather*}
for a (holomorphic) representative $u$ of $A$, and $\deg(q^{A+B})=\deg(q^A)+\deg(q^B)$. 
Choose an $\varepsilon <<1$, let $N$ be an arbitrary real number, and define the Novikov ring in two variables as
\begin{align*}
\Lambda^2_{(\varepsilon)}=\bigg\{ \sum_{i,j}c_{A_ij}q^{A_i}r^{\eta_j}\Vert & c_{A_ij}\in \bb{Z}_2; q^{A_i}\in \mc{Q},\eta_j\in\bb{R};\\ (1-\varepsilon)&\deg(q^A)+\eta_j\geq 0; \\& \#\{ a_{Aj}\neq 0\Vert \deg(q^{A_i})+ \eta_j\leq N\}<\infty \bigg\}
\end{align*}
where the multiplication is $q^A\cdot q^B:=q^{A+B}$ and extended to products and sums in the usual way.

We choose a Morse-Smale function $f$ on $L$ that is constructed by
$$f=\pi^*b+\sum_i \phi_i g_i$$
where $b$ is Morse-Smale on $L_B$ and $g_i$ are Morse-Smale on 
$$\bigsqcup_i L_{F,i}:=\pi^{-1}(\crit{b})\cap L;$$ the \emph{critical fibers}. We assume that the $g_i\vert_{L_{F,i}}$ 
 and $b$ have a unique index $0$ and index $n$ critical point, denote $x^M_i$ resp. $x^m_i$, where we define the index of a critical point as
$$\I (x)=\dim W^+_{X_f}(x),$$
the dimension of the stable manifold.
One can then choose an appropriate pseudo-gradient 
$$X_f=X_g\oplus X_b$$
in the splitting $TL_F\oplus L\cap H$ on $L$.

Define the \emph{Floer chain complex} as the free module generated by the critical points of $X_f$
$$CF^*(L,\Lambda^2):=\bigoplus_{x^i\in \crit{X_f}}\Lambda^2x^i $$
with grading 
\begin{gather*}
\vert x\vert =\dim W^+_{X_b}(\pi(x)) + \dim W^+_{X_g}(x)=\dim W^+_{X_f}(x)\\
\vert q^A\vert := \deg(q^A)\\
\vert r^\eta\vert =\eta
\end{gather*}
and extended on products so that
$$\vert x^iq^A r^\eta\vert=\vert x \vert+\eta+\vert q^A\vert.$$

Let $\rho\in\hom(\pi_1(L),\Lambda^{2\times})$ denote a choice of rank one local system. This $\hom$ space is a smooth manifold modeled on $H^1(L,\Lambda^{2\times})$ by Poincare duality and exponentiation $\exp:\Lambda^{2\times}\rightarrow \Lambda^2$. For a $P_\g$-holomorphic configuration $u:C\rightarrow (E,L)$, let $\hol_\rho(u)$ denote the evaluation of $\rho$ on $[u(\partial C)]$.

For a relative homology class $A\in H_2^\circ (E,L,\bb{Z})$ let 
$$a(A):=\int_\mc{C}a^*(u)$$
where $u:\mc{C}\rightarrow (E,L)$ is a holomorphic representative. Since $a\vert L\equiv 0$ [Lemma \ref{fiberedlagrangianlemma}] and $a$ is closed, it follows $a(A)$ is well-defined as a homotopy invariant.

%
%
%
%
%
%
%

Define the $\bb{Z}_2$-$A_\infty$ maps as

\begin{equation}\label{ainftydefmod2}
\nu^n_{\bb{Z}_2,\rho}(x_1,\dots, x_n):=\sum_{x_0,[u]\in\mc{M}_0(E,L,\mc{P},x_0,x_1,\dots,x_n)} \hol_\rho(u)q^{\pi_*A}r^{a(A)}x_0
\end{equation}

It is shown in [Theorem 6.1 \cite{fiberedpotential}] that a version of \ref{ainftydefmod2} with coefficients in the Novikov ring over $\bb{C}$ satisfies the $A_\infty$-axioms. The proof that 
$$\mc{A}(L)_{\bb{Z}_2}:=\bigg( CF(L,\Lambda^2_{\bb{Z}_2}),\nu_{\bb{Z}_2,\rho}^n\bigg)$$
satisfies these axioms is similar and follows from a description of the boundary of the $1$-dimensional moduli spaces from \ref{compactnessthm}.

\subsection{Invariance}\label{invariancesection}
It is important to see how the definition \ref{ainftydefmod2} respects homotopy of almost complex structure and gauge transformation. This is discussed in \cite{floerfibrations} section 9, and we review the discussion here.

Given two coherent, regular perturbation data $P_1$ and $P_2$ in the rational case \cite{CW2} and their respective $A_\infty$-algebras $A_1$, $A_2$, one defines a perturbation morphism $P_{01}:A_1\rightarrow A_2$ based on a count of quilted Floer configurations which are $P_1$-holomorphic on one side of a ``seam" and $P_2$-holomorphic on the other (see section 3.1 \cite{CW2} for the precise definition of an $A_\infty$ morphism). In order to show that $P_{10}\circ P_{01}$ is $A_\infty$ homotopic to the identity, one defines a moduli of twice-quilted disks and shows that the appropriate count gives rise to an $A_\infty$ homotopy between the two maps.

In our case, the divisor $D$ is not stabilizing for $L$, so in order to related the above objects to invariants in the rational case, we first observe that there is a natural map
\begin{gather*}\mf{f}:\Lambda^2_{\bb{Z}_2}\rightarrow \Lambda_{t,\bb{Z}_2}\\
q^\eta r^\rho\mapsto t^{\eta+\epsilon\rho}
\end{gather*}
where the $\epsilon$ is that in the weak coupling form. Assuming that $(E,L)$ is a rational pair, let $D^1\subset E$ be a stabilizing divisor for $L$ which intersects $\pi^{-1}(D_B)=:D^0$ $\epsilon$-transversely and is $\theta$-approximately holomorphic by [Theorem 8.1 \cite{CM}] and [Theorem 3.6 \cite{CW1}] respectively. Then we define a perturbation morphism between our $A_\infty$-algebra and a more general one by counting quilted Floer configurations that are $\pi$-adapted to $D^0$ below the seam and adapted to $D^1$ at or above the seam, while fixing neighborhoods or the boundary and special points on the domain on the seam component where the almost complex structure is constant and adapted to both $D^0$ and $D^1$. To show that the composition of such a perturbation morphism and a morphism in the reverse direction is $A_\infty$ homotopic to the identity, one adapts the construction of twice-quilted Floer trajectories \cite{CW2} in an analogous way.
\section{A potential formula}\label{examplesection}
We will see in example \ref{realprojectivebundles} that the potential for the total Lagrangian depends on more than the potential for the base and fiber, and requires the knowledge of almost all of the holomorphic representatives in the base. Such a thing is hard to compute outside of some very special cases. We include a family of examples.

Let $\bb{P}^m\rightarrow \bb{P}(\mc{V})\rightarrow B$ be a complex projective symplectic fibration of dimension $2m+2n$, and $L_B\subset B$ a rational Lagrangian such that $\L_B=\t{fix}(\tau_B)\neq \emptyset$ for an anti-symplectic involution. Suppose there is an anti-symplectic involution $\tau$ on $E$ such that
$$ \tau_B\circ \pi=\pi\circ \tau$$
and that $\tau$ is anti-holomorphic on the fibers when we identify $F_p$ and $F_{\tau(p)}$ holomorphically. We get a Lagrangian
$$L_F\rightarrow L\rightarrow L_B$$
and we aim to compute Floer type invariants for $L$ in some very special cases. Via the Arnold-Givetal conjecture \cite{givag}, one expects the Floer cohomology to be non-zero when defined. The conjecture states the following: Let $\phi^t$ with $t\in[0,1]$ be a Hamiltonian isotopy of $(M,\omega)$. Then
\begin{equation}\label{agestimate}\#\phi^1(L)\cap L\geq \sum_i \dim H^i(L,\bb{Z}_2).
\end{equation}
This conjecture is verified in many situations, starting with $(\bb{CP}^n,\bb{RP}^n)$ \cite{givag}\cite{ohrpn} and the case of real forms in monotone Hermitian symmetric spaces \cite{ohagconjecture}. Symplectic quotients were considered in \cite{frauenfelderag}, and the relative case was considered in \cite{iriyehsakaitasaki}.

\subsection{Example: Rank 1 real projective bundles}\label{realprojectivebundles}
We take $\bb{CP}^1\rightarrow \bb{CP}(\mc{O}\oplus \mc{O}_k)\rightarrow \bb{CP}^n$ for $n$ odd and find an anti-symplectic involution.
Let $U\subset \bb{CP}^n$. Recall the definition of the $k^{th}$ twisting sheaf
$$\mc{O}_k(U):=\bigg\lbrace (V,f)\in U\times \mc{A}(U)\vert f(a\cdot \bold{z})=a^k\cdot f(\bold{z})\, \forall a\in V \bigg\rbrace $$
where $\mc{A}(u)$ is the ring of homogeneous polynomials on $U$. For such a space, there is an anti-holomorphic involution
$\tau(V,f)=(\bar{V},\bar{f})$ that lifts the involution on the base, where $\bar{f}$ denotes the map
\begin{gather*}
\bar{\cdot}:\mc{A}(U)\rightarrow \mc{A}(\bar{U})\\
f(\cdot)\mapsto \bar{f}(\bar{\cdot})
\end{gather*}
Such a space can be realized as a toric variety, and hence a symplectic manifold by Delzant's construction. The above involution acts as anti-symplectic with regards to the toric symplectic form, and hence preserves the toric connection $T$ on the fibration. To the toric connection we can apply 6.4.1 from \cite{ms1} to obtain a weak coupling form:
$$\omega_{T,K}:=\varepsilon a_T+\pi^*\omega_{KKS}$$
to which $\tau$ is also anti-symplectic.

The above involution extends to an involution on $\mc{O}\oplus \mc{O}_k$ and hence to $E_{n,k}:=\bb{P}(\mc{O}\oplus\mc{O}_k)$. On $E_{n,k}$ the involution descends to one on $\bb{CP}^n$, so 
$$L:=\t{Fix}(\tau)$$
is an $\bb{RP}^1$ fibration over $\bb{RP}^n$.
The Maslov index $n+1$ disks in $(\bb{CP}^n,\bb{RP}^n)$ are in the class $A$ such that $2A$ is the class of a line. By an argument due to Oh [Lemma 4.6 \cite{ohrpn}] the toric complex structure $J_{\bb{CP}^n}$ is regular for holomorphic disk configurations with boundary in $\bb{RP}^n$. Moreover, the discussion in section 3.1 \cite{agm} and the proof of Theorem 3.6 in \cite{CW1} shows that we can find a smooth stabilizing divisor $D_B$ for $\bb{RP}^n\subset \bb{CP}^n$ which is a complex submanifold with respect to $J_{\bb{CP}^n}$. Theorem \ref{transversalitythm} holds for such a divisor and complex structure on the base.

For a $J_{\bb{CP}^n}$-holomorphic sphere $u$ in the class $2A$, it follows that $u^*E\cong \bb{P}(\mc{O}\oplus \mc{O}_k)$ as a holomorphic-bundle-with-connection. If the pullback connection is to be compatible with the transition maps on this Hirzebruch bundle then the holonomy around the equator must be a rotation by $k/2$. Thus in a smooth trivialization-with-connection $\pi\circ u^*L$ is the set
\begin{equation}
L:=\left\lbrace \Big( [x_0,e^{ki\theta/2}x_1],\theta \Big): x_i\in \bb{R}\right\rbrace \subset \bb{CP}^1\times S^1
\end{equation}
and the holonomy around a closed loop $\gamma$ in $\bb{P}^1$ is rotation by the angle $2\pi k\cdot \mc{A}(\gamma)$ where $\mc{A}(\gamma)$ is the area enclosed by $\gamma$ (assuming the area of the sphere is $1$). 
For $k$ even, there are some easy sections in the given Hirzebruch trivialization:
\begin{gather}
\hat{v}_{[0,1]}(z)=[0,1]\label{constantsectiontop}\\
\hat{v}_{[1,0]}(z)=[1,0]\label{constantsectionbot}\\
\hat{v}_{[x_0,x_1]}(z)=[x_0,z^{k/2}\cdot x_1] \qquad x_0\neq 0 \label{non-constsection}
\end{gather}
Presumably these are the "covariant-constant" sections from Theorem \ref{liftthm}. The case of \ref{non-constsection} shows how the gauge transformation from corollary \ref{heatflowgaugever} need not preserve the curvature of a section in the projectivization.
The vertical Maslov indices are as follows:
\begin{gather}
\mu(\hat{v}_{[0,1]})=-k\\
\mu(\hat{v}_{[1,0]})=k\\
\mu(\hat{v}_{[x_0,x_1]})=k
\end{gather}
which can be obtained by taking a global trivialization over the disk centered at one of the poles of $\bb{CP}^1$ (c.f. Lemma 6.2 \cite{salamonakveld}). Note that the maps \ref{constantsectiontop} and \ref{constantsectionbot} and hence these computations make sense when $k$ is odd.

Next we compute lifts of $v$ that are of expected dimension $0$ using the base a.c.s. $J_{\bb{CP}^n}$, integrable invariant fiber structure and following remark \ref{computeremark}. Let $v_{F_1}:(D,\partial D)\rightarrow (\bb{CP}^1,\bb{RP}^1)$ be a holomorphic map from the disk to a hemisphere in the fiber $F_{v(1)}$. 
In the trivialization from Theorem \ref{liftthm} let $\tilde{v}^m_{[0,1]}$ be the lift of an $m$-fold branched cover of $v$ as in \ref{constantsectionbot}, and let $$\hat{v}^m_{[0,1]}:(D,\partial D)\rightarrow  (v^*E,v^*L)$$ denote the image of this section under inverse trivialization followed by gauge transformation perturbation from Theorem \ref{liftthm}. It follows that the Maslov index of such a disk is
$$\mu(\widehat{v}^{m}_{[0,1]})=m(n+1-k)$$
Denote by 
 $$\widehat{v^{m,l}_p}:C\rightarrow (E,L)$$ the holomorphic configuration that consists of $\hat{v}^m_{[0,1]}$ together with an $l$-fold holomorphic cover the of a hemisphere
$$v_{F_p}^l:(D,\partial D)\rightarrow (\bb{CP}^1,\bb{RP}^1)$$ in the fiber containing $p$. We count these types of configurations that are isolated according to index \ref{indexformulatotalspace}, and which pass through a generic point $p\in L$. As already mentioned, there is a fiber hemisphere map $v_{F_p}$ that evaluates $1$ to $p$, and we count these up to reparameterizations fixing $1$.
Any non-constant holomorphic section of $v^*E$ has vertical maslov index $k$ if $k$ is even. Thus the total Maslov index of such a disk is
$$n+1+k$$
so we would not count this disk on its own.
As discussed in example 3.2 \cite{aurouxspeciallagfibrations}, it is to our benefit to count isolated configurations that contain the exceptional section \ref{constantsectionbot}. Such a section gives a holomorphic disk in the total space $(E,L)$ of Maslov index
$$n+1-k.$$
There are several cases depending on the positive integer values of $k$ and $n$. We don't consider when $k=0$.
\begin{enumerate}
\item[Case $n+1-k=2$] Such a case has expected dimension $0$ but the evaluation is not a generic point.

\item[Case $n+1-k=0$]
We include the configuration that consists of the fiber disk together with the exceptional section $\hat{v}_{[0,1]}^1$, denoted as a parameterization class of $\hat{v}_{[0,1]}^{1,1}$. As transversality holds for multiple covers $\hat{v}_{[0,1]}^m$, the count includes $v_{[0,1]}^{m,1}$ for $m\in \bb{Z}_{\geq 0}$.

\item[Case $n+1-k=1$] We include the fiber disk, and separately the double cover $\widehat{v^{2,0}_{[0,1]}}$.

\item[Case $n+1-k<0$] This case counts non-trivial positive solutions $(m,l)$ to the integer equation
\begin{equation} m \left( n+1-k\right) + 2l= 2
\end{equation}
for when to count the maps $\widehat{v_{[0,1]}^{m,l}}$. 
\end{enumerate}
Let $\mc{I}_{\geq 0}$ denote the set of non-negative integer solutions, and denote by $y_1$ resp. $y_2$ denote the values of a representation $\rho\in \hom (\pi_1(L),\Lambda^{2 \times})$ on the generators of $\pi_1(L)$. If $\mc{D}_L (\rho)$ denotes the number mod $2$ of Maslov index $2$ treed disks through a generic point weighted by their symplectic area and representation $\rho$, then we have
\begin{align}\label{potentialforrank1}
\mc{D}_L &(\rho)=\left(y_1+y_1^{-1}\right)r^{a[v_{F_p}]}\\ +&\sum_{(m,l)\in \mc{I}_{\geq 0}}\left(y_1^ky_2^m-y_1^{-k}y_2^m+y_1^ky_2^{-m}-y_1^{-k}y_2^{-m}\right)q^{m\omega_B[\hat{v}_{[0,1]}]}r^{la[v_{F_p}]+ma[\hat{v}_{[0,1]}]}
\end{align}
In particular, picking the trivial representation gives $0$.

\subsection{Example: Rank $s$ real projective bundles}
In analogy we can consider fixed point sets of anti-holomorphic involutions in the following
$$\bb{CP}^k\rightarrow \bb{P}\Big(\mc{O}\oplus \mc{O}_{k_1}\oplus\dots \mc{O}_{k_s}\Big)\rightarrow \bb{CP}^n$$

Assume that the $k_i\leq k_{i+1}$

From the discussion at the beginning of example \ref{realprojectivebundles} it follows that there is a Lagrangian with respect to a coupling form that fibers as
$$\bb{RP}^s\rightarrow L\rightarrow \bb{RP}^n$$
We adopt the same notation as the previous example, with $H$ the toric connection on $E$, $D_B$ the stabilizing divisor in the base.
Take $J_F$ to be the fiberwise integrable complex structure. In the connection $H$ we have some holomorphic sections
\begin{align}
\hat{v}_{[1,a_1,\dots a_s]}(z):=&[1,z^{k_1/2}a_1,\dots z^{k_s/2}a_s]\label{topsection}\\
\label{bottomsection} \hat{v}_{[0,\dots a_i,\dots a_s]}(z):=&[0,\dots 0,a_{i},\dots a_{j},z^{\frac{k_{j+1}-k_j}{2}}a_{j+1},\dots z^{\frac{k_s-k_i}{2}}a_s], \quad k_i=k_{j},\, (a)_i^j\in \bb{R}^{j-i}\setminus 0
\end{align}
where is it understood that for $k_f$ or $k_f-k_e$ odd we have $a_f=0$. The vertical Maslov indices are
\begin{align*}
\mu(\hat{v}_{[1,a_1,\dots a_s]})=&\sum_{t=1}^s k_t\\
\mu(\hat{v}_{[0,\dots a_i,\dots a_s]})=&\sum_{t=1}^s k_t - (s+1)c \cdot k_i \\
\end{align*}
In addition, there are some other sections incorporating fiber classes; namely let $f^i(z):D\rightarrow \bb{H}$ be the M\"obius transformation  that sends $\partial D$ to $\bb{R}$. Denote $f^i_l(z)=f^i(z^l)$. We consider the following modification to the section \ref{bottomsection}
\begin{equation}\label{bottomsectionmods}
\hat{v}^{\underline{l}}_{[0,\dots a_i,\dots a_s]}(z):=[0,\dots 0,f^i_{l_i}a_i,\dots f^j_{l_j}a_j, z^{k_{j+1}-k_i}f^{j+1}_{l_{i+1}}a_{j+1},\dots,z^{k_s-k_i}f^{s}_{l_s}a_s]
\end{equation}

Such a section has Maslov index 
$$\sum_{t=1}^s k_t-(s+1)k_i + 2\vert\underline{l}\vert$$
With all of this said, the section \ref{bottomsection} and modifications \ref{bottomsectionmods} do not intersect a generic point in $v^*L$, so they only show up in configurations that already involve the fiber class. Summarily we only consider the above sections $\hat{v}$ for which 
$$\mu(\hat{v})\leq -n-1$$
To understand which of these sections are counted we write down parameterizations of the fiber class through a generic point $p=[a_0,\dots a_s]$ with $a_i\neq 0$. One can show  that any holomorphic disk with $f(1)=p$ is of the form
\begin{equation}\label{verticaldisksforranks}
f(z)= [\phi_0 a_0,\phi_1 a_1,\dots \phi_sa_s]
\end{equation}
where each $\phi_i:(D,\partial D)\rightarrow (\bb{C},\bb{R})$ is a finite Blaschke product with $\phi_i(1)=1$
[c.f. Theorem 10.1 \cite{cho}] (and not all $\phi_i=0$ at a given $z$). Such a map has Maslov index
$$(s+1)\sum_{i=1}^s\zeta(\phi_i)$$
where $\zeta$ is the number of Blashke factors.

Taken together, the count will be a subset of integer solutions to the system
\begin{align}\label{indexcountequation}
m\left(\sum_{t=1}^s k_t-(s+1)k_i +n+1\right) + (s+1)\sum_{i=1}^s\zeta(\phi_i)+2\vert\underline{l}\vert=2
\end{align}
for each $i$ in the variables $\underline{l}$, $m$, and $\zeta(\phi_i)$.

\

 \subsection{Computing Floer cohomology of involution fixed point sets in projectivized vector bundles}
Based on the observations in the above examples, we consolidate a general argument that allows us to show that the $A_\infty$-algebra is unobstructed:
\begin{theorem}\label{agthm}
Let $\bb{P}(\mc{V})\rightarrow \bb{CP}^n$ be the projectivization of a vector bundle with a lift $\tau$ of the anti-symplectic involution on $\bb{CP}^n$. Then for $L=\text{Fix}(\tau)$, $HF(L,\Lambda_{t,\bb{Z}_2})$ is defined and isomorphic to $H^{Morse}(L,\Lambda_{t,\bb{Z}_2})$.
\end{theorem}

\begin{proof}
Most of the proof is demonstrated through example \ref{realprojectivebundles}, but we provide it here. For the base, all holomorphic disks come in a pairs in the integrable toric structure $J_{\bb{CP}^n}$, which is regular by a doubling trick followed by a topological argument [Lemma 4.6 \cite{ohrpn}]. By section 3.1 in \cite{agm} and by the proof of [Theorem 3.6 in \cite{CW1}], we can find a high degree smooth hypersurface $D_B\subset\bb{CP}^n$ which is holomorphic in the toric structure and is stabilizing for $\bb{RP}^n$. We apply Theorem \ref{liftthm} using this base complex structure and stabilizing divisor to find a comeager subset of regular gauge transformations for each section class. By Remark \ref{computeremark} is suffices to find sections in the given connection, and then perturb using a regular (invertible) gauge transformation.

Given a holomorphic configuration $u$ in the perturbation system above, the involution $\tau$ gives another configuration $u(z)\mapsto \tau\circ u(\bar{z})=:v(z)$ with matching boundary conditions. To see that $v$ is holomorphic it suffices to see that $\tau$ is anti-holomorphic. Note that in a local trivialization the almost complex structure
$$J:=\begin{bmatrix}
J_{\bb{P}(V)} & J_{\bb{P}(V)}\circ X_\sigma-X_\sigma\circ J_{\bb{CP}^n} \\
0 & J_{\bb{CP}^n}
\end{bmatrix}
$$
is the unique such one from Lemma \ref{uniqueacs} when the base is $\bb{CP}^n$. We have $\tau^*a=-a$, so $\tau$ sends $TF^{\perp a}$ to $TF^{\perp a}$. Thus $\tau^* J$ is the unique almost complex structure that agrees with $-J_{\bb{P}(V)}$ on the fibers, preserves $TF^{\perp a}$ and makes $\pi$ holomorphic w.r.t. $-J_{\bb{CP}^n}$, so it must be equal to $-J$ in the above coordinates, and hence globally.

Thus, the full potential vanishes at the trivial representation, so that the definition of Floer cohomology of $L$ is unobstructed. Moreover, any holomorphic disk in the Floer differential cancels mod 2 for exactly the same reason. Hence
$$HF(L,\Lambda_{\bb{Z}_2,t},J,\mc{G})\cong H^{Morse}(L,\Lambda_{\bb{Z}_2,t}^2)$$
From the discussion in the invariance section \ref{invariancesection} we have an isomorphism for a generic perturbation system. Since $\text{rank}_{\Lambda_{\bb{Z}_2,t}}HF(L,\Lambda_{\bb{Z}_2}^2)\leq \text{Fix}(\phi)$ for a Hamiltonian isotopy $\phi$, estimate \ref{agestimate} follows.
\end{proof}

\printbibliography

\end{document}